\documentclass[12pt,fleqn]{amsart}

\usepackage[latin1]{inputenc}
\usepackage[english]{babel}
\usepackage{indentfirst}
\usepackage{amssymb}
\usepackage{amsthm}
\usepackage{xcolor}
\usepackage[all]{xy}
\usepackage[mathscr]{eucal}
\usepackage{hyperref}

\usepackage{color}
\usepackage{enumerate}
\usepackage{amssymb}
\newcommand{\cA}{\mathcal A}

\newcommand{\cC}{\mathcal C}
\newcommand{\cD}{\mathcal D}

\newcommand{\cF}{\mathcal F}
\newcommand{\cG}{\mathcal G}

\newcommand{\fA}{\mathfrak A}
\newcommand{\fB}{\mathfrak B}

\newcommand{\ult}{\protect{ult}}
\newcommand{\clop}{\protect{\rm Clop}}
\newcommand{\la}{\langle}
\newcommand{\ra}{\rangle}
\newcommand{\sub}{\subset}
\newcommand{\eps}{\varepsilon}
\newcommand{\er}{\mathbb R}
\newcommand{\sm}{\setminus}

\newcommand{\vf}{\varphi}
\newcommand{\qu}{\mathbb Q}
\newcommand{\core}{\protect{\rm core}}

\newcommand{\N}{\mathbb{N}}
\newcommand{\erre}{\mathbb{R}}
\def\epsilon{\varepsilon}

\newcommand{\To}{\longrightarrow}
\newcommand{\rto}{\rightarrow}

\newtheorem{theorem}{Theorem}[section]
\newtheorem{proposition}[theorem]{Proposition}
\newtheorem{corollary}[theorem]{Corollary}
\newtheorem{lemma}[theorem]{Lemma}

\newtheorem{problem}[theorem]{Problem}

\theoremstyle{definition}
\newtheorem{definition}[theorem]{Definition}
\newtheorem{example}[theorem]{Example}
\newtheorem{remark}[theorem]{Remark}

\newcommand{\ueberlein}{\protect{\mbox{\sc UniEberlein}}}
\newcommand{\ale}{\protect{\mbox{\sc UE(1)}}}
\newcommand{\singletonsupported}{\protect{\mbox{\sc UE(1)}}}
\newcommand{\Gulko}{\protect{\mbox{\sc Gul'ko}}}
\newcommand{\Talagrand}{\protect{\mbox{\sc Talagrand}}}

\newcommand{\eberlein}{\protect{\mbox{\sc Eberlein}}}

\newcommand{\ccc}{\protect{\mbox{\sc ccc}}}
\newcommand{\separable}{\protect{\mbox{\sc separable}}}
\newcommand{\SPM}{\protect{\mbox{\sc spm}}}
\newcommand{\corson}{\protect{\mbox{\sc Corson}}}
\newcommand{\caloo}{\protect{\mbox{\sc cal}(\omega_1)}}
\newcommand{\wrn}{\protect{\mbox{\sc WRN}}}
\newcommand{\rn}{\protect{\mbox{\sc RN}}}
\newcommand{\scattered}{\protect{\mbox{\sc scattered}}}

\newcommand{\zerodimensional}{\protect{\mbox{\sc zerodimensional}}}

\newcommand{\wrnb}{\protect{\mbox{\sc WRN(B)}}}

\def\Clop{\text{$\operatorname{Clop}$}}
\def\ult{\text{$\operatorname{ult}$}}

\numberwithin{equation}{section}

\title[Abundance of independent sequences]{Abundance of independent sequences \\in compact spaces and  Boolean algebras}
\author[A. Avil\'{e}s]{Antonio  Avil\'{e}s}
\address{Departamento de Matem\'{a}ticas\\
Facultad de Matem\'{a}ticas\\ Universidad de Murcia\\ 30100 Espinardo, Murcia\\
Spain} \email{avileslo@um.es}
\author[G.\ Mart\'{\i}nez-Cervantes]{Gonzalo Mart\'{\i}nez-Cervantes}
\address{Departamento de Matem\'{a}ticas\\
Facultad de Matem\'{a}ticas\\ Universidad de Murcia\\ 30100 Espinardo, Murcia\\
Spain}
 \email{gonzalo.martinez2@um.es}
\author[G. Plebanek]{Grzegorz Plebanek}
\address{Instytut Matematyczny\\ Uniwersytet Wroc\l awski\\ Pl.\ Grunwaldzki 2/4\\
50-384 Wroc\-\l aw\\ Poland}
\email{grzes@math.uni.wroc.pl}


\thanks{A. Avil\'es and  G. Mart\'inez-Cervantes were partially supported by project MTM2017-86182-P (Government of Spain, AEI/FEDER, EU) and by Fundaci\'{o}n S\'{e}neca, ACyT Regi\'{o}n de Murcia under project 20797/PI/18. The research of G. Mart\'inez-Cervantes has been co-financed by the European Social Fund (ESF) and the Youth European Initiative (YEI) under the Spanish S\'eneca Foundation (CARM) (ref. 21319/PDGI/19).
G. Plebanek was partially supported  by the grant 2018/29/B/ST1/00223 from National Science Centre, Poland.}

\begin{document}

\begin{abstract}
It follows from a theorem of Rosenthal that a compact space is $ccc$ if and only if every Eberlein continuous image is metrizable.
Motivated by this result, for a class of compact spaces $\cC$ we define its orthogonal $\cC^\perp$ as the class of all compact spaces for which every continuous image in $\cC$ is metrizable. We study how this operation relates classes where centeredness is scarce with classes where it is abundant (like Eberlein and $ccc$ compacta), and also classes where independence is scarce (most notably weakly Radon-Nikodým compacta) with classes where it is abundant. We study these problems for zero-dimensional compact spaces with the aid of Boolean algebras, and show the main difficulties arising when passing to the general setting. Our main results are the constructions of several relevant examples.
\end{abstract}

\maketitle

\section{Introduction}

The class of weakly Radon-Nikod\'{y}m ($\wrn$) compact spaces
was introduced by Glasner and Megrelishvili  \cite{GM12} as a natural superclass of
the well-studied  Radon-Nikod\'{y}m compacta.
A compact space is weakly Radon-Nikod\'{y}m if it is
 homeomorphic to a $weak^\ast$ compact subset  of $X^\ast$, where
 $X$ is a  Banach space containing no isomorphic copy of $\ell_1$.
 There is a useful purely combinatorial characterization of $\wrn$ (see \cite[Theorem 2.1]{Ma17})
 and it reads as follows.

\begin{theorem}\label{Gonzalo}
	A compact space $K$ is $\wrn$ if and only if $K$ can be embedded into some
cube $[0,1]^\kappa$ in such a way that for every $p<q$ the family of pairs
\[(\{x\in K : x(\alpha)\leq p\},\{x\in K : x(\alpha)\geq q\})_{\alpha}\]
contains no infinite independent sequence.
\end{theorem}

Recall that  a sequence of pairs of disjoint sets $(A^0_n,A^1_n)$ is
said to be
\textit{independent} if
$A^{\varepsilon_1}_1\cap A^{\varepsilon_2}_2\cap\cdots \cap A^{\varepsilon_m}_m\neq\emptyset$ for every $m$  and any choice of $\varepsilon_1,\ldots,\varepsilon_m\in\{0,1\}$.
Accordingly, a sequence of subsets $A_n$ of a space $X$ is independent if
the pairs $(A_n, X\sm A_n)$ are independent.

 If $K$ is Eberlein compact (that is, $K$ is homeomorphic to a weakly compact subset of some Banach space)
 then, by the classical Amir-Lindenstrauss theorem, $K$ embeds into some cube
 $[0,1]^\kappa$ in such a way that the family of sets $\{x\in K : x(\alpha)\geq q\}$
 contains no infinite centered sequence whenever $q>0$.
 Recall that  a family of sets is said to be centered if any finite intersection of its elements is nonempty.


 From this point of view, there is certain parallelism between
 Eberlein compacta, in which centered sequences are scarce,
 and weakly Radon-Nikod\'{y}m compacta,  in which independent sequences are scarce.
 This analogy can be made more transparent when we confine ourselves  to the zero-dimensional case.
 Let us say that a Boolean algebra $\fA$ is weakly Radon-Nikod\'{y}m
 (Eberlein) if the Stone space of $\fA$ is weakly Radon-Nikod\'{y}m compact (Eberlein compact, respectively).

\begin{theorem}\label{intro:wrn}
	A Boolean algebra $\mathfrak{A}$ is weakly Radon-Nikod\'{y}m if and only if
$\fA$ is generated by a family $\cG\sub\fA$ which can be written as $\cG=\bigcup_n\cG_n$
so that no $\cG_n$ contains an infinite independent sequence \cite[Proposition 3.2]{AMP18}.

A Boolean algebra $\mathfrak{A}$ is Eberlein if and only if
$\fA$ is generated by a family $\cG\sub\fA$ which can be written as $\cG=\bigcup_n\cG_n$
so that no $\cG_n$ contains an infinite centered sequence \cite[Remark 2, pg.~107]{Ro74}.
\end{theorem}

In a sense, weakly Radon-Nikod\'{y}m compacta and Eberlein compacta  are closely related to the scarcity of independence or centeredness;
   how do we describe the abundance of those two properties?
   In the latter case the answer is the well known world of chain conditions,
   the most basic of which is the countable chain condition ($ccc$).
    A Boolean algebra is $ccc$ if it contains no uncountable pairwise disjoint family,
     which is equivalent to saying that every uncountable subfamily
     contains an infinite centered family.
     A similar definition applies to compact spaces by considering families of open sets.
     Playing a little with the definitions and using a classical theorem of Rosenthal (see
     section \ref{sectionccc}), a kind of duality becomes apparent.

\begin{theorem}\label{Eberleinorthogonal}
	A Boolean algebra $\mathfrak{A}$ is $ccc$ if and only if every Eberlein subalgebra of
$\mathfrak{A}$ is countable. A compact space $K$ is $ccc$ if and only if every Eberlein continuous image of $K$ is metrizable.
\end{theorem}

This inspires a research program around the following items.

\begin{enumerate}[(1)]
	\item There is an analogy between conditions involving independence and conditions involving centeredness. While the abundance of centeredness (chain conditions) is well studied, its  parallel
world of abundance of independence is not so much.
	\item The abundance and scarcity classes can be studied in duality:
Given a class $\mathcal{C}$ of one kind, we can create the \emph{orthogonal} class $\mathcal{C}^\perp$ of the other kind. In the topological setting, $\mathcal{C}^\perp$ would be the class of compact spaces $K$ whose all continuous images in $\mathcal{C}$ are metrizable. In the Boolean setting, the class of Boolean algebras whose all subalgebras in $\mathcal{C}$ are countable.
	\item To what extent the Boolean and the topological settings differ? Is it possible to reduce the study of compact spaces in these classes to the case of zero-dimensional compacta through taking subspaces or continuous images?
\end{enumerate}

In this paper we  make first steps towards exploring  these ideas.
Section~\ref{occ} shows some basic facts about the behavior of the operation of taking orthogonals of a class.
In Section~\ref{sectionccc} we stay in the centered world,
examining the orthogonal classes there.
Using the concept of orthogonality, Theorem~\ref{Eberleinorthogonal} above states
that $\eberlein^\perp = \ccc$.  Indeed, the orthogonal class of many classical classes of compact spaces studied in the Banach Space setting coincides with the class of $ccc$ compact spaces (see Proposition \ref{occ:2}). In general, as we shall see in Section \ref{sectionccc}, one can interpret a large number of (essentially) known results in this manner.

In the subsequent sections we investigate
the class of compact spaces $\wrn^\perp$ and the related class $\wrnb^\perp$ of Boolean algebras. In particular,
in Section \ref{sectionOrthogonalWRN} we provide a simple characterization of Boolean algebras in $\wrnb^\perp$, whereas in Section \ref{SectionCorsonandWRNB} we construct, under $MA_{\omega_1}$, an example of a nonmetrizable zero-dimensional Corson compact space whose clopen algebra belongs to $\wrnb^\perp$. The difficulties in checking whether this example belongs to $\wrn^\perp$ motivate us to study the relation between $\wrn^\perp$ and $\wrnb^\perp$.
Section \ref{between} is mainly devoted to the construction of a zero-dimensional nonmetrizable compact space whose clopen algebra belongs to $\wrnb^\perp$ but having a nonmetrizable WRN continuous image (so not in $\wrn^\perp$).

One can also consider a natural stronger version of orthogonality, 
requiring compact spaces to be hereditarily orthogonal to the original class. 
In the final section we 
deal with the strong orthogonal of the class of zero-dimensional compacta, previously
investigated by Piotr Koszmider. Here the basic problem is whether such an orthogonal class
is nontrivial, i.e.~if there is a nonmetrizable compactum $K$ such that every continuous image of $K$
contains no nonmetrizable zero-dimensional compact subspaces. 
Koszmider  \cite{Kosz16} presented several consistent constructions of such spaces $K$ discussing
some additional properties they may have.
Building on Koszmider's ideas, in Section \ref{ocZDs} we make our modest contribution to the subject
showing that, under various set-theoretic assumptions, these examples can be 
weakly Radon-Nikod\'ym.

\section{Orthogonal classes of compact spaces}\label{occ}

For a class $\cC$ of compact spaces we introduce two orthogonal classes
as follows.

\begin{definition}
\begin{enumerate}[(a)]
\item $\mathcal{C}^\perp$ is  the class of those compact spaces $K$ such that every
continuous image  of $K$  that
belongs to $\mathcal{C}$ is metrizable;
\item $\mathcal{C}^{(\perp)}$ is the class of those compact spaces $K$ such that every
continuous image  of any closed subspace of $K$  that belongs to $\mathcal{C}$ is metrizable.
\end{enumerate}
\end{definition}

Obviously, $\mathcal{C}^{(\perp)}\subseteq \mathcal{C}^{\perp}$ for every class of compact spaces $\cC$.
Note that if $\cC\sub\cD$ then $\cD^\perp\sub\cC^\perp$. We also record the following general facts.

\begin{lemma}\label{occ:1}
	Let $\cC$ be a class of compact spaces that is closed under continuous images. Then
	
	\begin{enumerate}[(i)]
		\item $\cC\sub\cC^{\perp\perp}$;
		\item $\cC^\perp = \cC^{\perp\perp\perp}$.
	\end{enumerate}
\end{lemma}

\begin{proof}
	If $K\notin \cC^{\perp\perp}$ then $K$ has a nonmetrizable continuous image $L\in\cC^\perp$.  Then, in particular, $L\notin \cC$ so
	$K\notin\cC$; this proves $(i)$.
	
	Note that $\cC^\perp$ is automatically closed under continuous images.  Hence, $\cC^\perp \sub  \cC^{\perp\perp\perp}$ by $(i)$.
	For the reverse inclusion, apply  $(i)$ to the previously mentioned fact that $\cC\sub\cD$ implies $\cD^\perp\sub\cC^\perp$.
\end{proof}

Although $\cC^\perp$ is always closed under continuous images, it might not be closed under subspaces. Nevertheless, $\mathcal{C}^{(\perp)}$ is always closed under subspaces and continuous images by the very definition and
the following fact, which is analogous to \cite[Lemma 4.1]{Kosz16}.

\begin{lemma}
	For any class $\mathcal{C}$ of compact spaces, $\cC ^{(\perp)}$ coincides with the class of those compact spaces $K$ such that every subspace of every continuous image of $K$ is metrizable whenever it is in $\cC$.
\end{lemma}
\begin{proof}
	It is immediate that every subspace of any continuous image of $K$ is a continuous image of a subspace of $K$.
On the other hand, if $K_1$ is a subspace of $K$ and $f:K_1 \rightarrow L$ is a
continuous surjection then,
assuming that  $L$ is embedded in $[0,1]^\kappa$ for some $\kappa$,
by the Tietze extension theorem,
there is a continuous extension $\widehat{f}:K \rightarrow [0,1]^\kappa$.
Consequently,  $L$ is a subspace of a continuous image of $K$.
\end{proof}

We finish this section with the following elementary fact.

\begin{lemma}
\label{lemhereditarilyCperp}
For any class $\mathcal{C}$ of compact spaces,
$\cC ^{(\perp)}$ is the class of hereditarily $\cC^\perp$ compact spaces, i.e.~the class of those compact spaces $K$ such that every closed subspace of $K$ belongs to $\cC^\perp$.
\end{lemma}

\begin{proof}
If $K\in \cC ^{(\perp)}$ and $L \subseteq K$ is a subspace, then every continuous image of $L$ belonging to $\cC$ is metrizable, so $L \in \cC^\perp$.

On the other hand, if $K$ is  hereditarily $\cC^\perp$ and $L \subseteq K$ is a subspace, then $L \in \cC^\perp$ and therefore any continuous image of $L$ belonging to $\cC$ is metrizable, so $K \in \cC ^{(\perp)}$.
\end{proof}

In the next sections we study the classes $\cC^\perp$ and  $\cC ^{(\perp)}$ for several well known classes of compact spaces $\cC$. Most often, we focus on the class $\cC^\perp$ and we do not pay too much attention to the class $\cC ^{(\perp)}$. This is due to the fact that, once we have a characterization for the class $\cC^\perp$, the best characterization for the class $\cC ^{(\perp)}$ that we are able to obtain is the one given by a direct application of Lemma \ref{lemhereditarilyCperp}. Nevertheless, sometimes this characterization for the class $\cC ^{(\perp)}$ is obscure and even not  useful enough to easily determine whether this class is nontrivial.
This will be the case for $\cC$ the class of zero-dimensional compacta (see Section \ref{ocZDs})

\section{Orthogonal classes and ccc compacta}\label{sectionccc}

 Rosenthal \cite{Ro69} proved that a compact space $K$ is $ccc$
if and only if every weakly compact subset of the Banach space $C(K)$ is separable.
This classical result may be translated to saying that Eberlein $ccc$ compact spaces are metrizable.
 Recall that an Eberlein compactum (resp.~uniformly Eberlein compactum) is
  a compact space homeomorphic to a weak compact subset of a Banach space (a Hilbert space, respectively).
For all undefined classes of compacta that are mentioned below
we refer the reader to Negrepontis \cite{Ne84} or Fabian  \cite{FabianGD}.

We denote classes of compacta by their names written in small capitals so, for example,
  {\sc Eberlein} ($\ueberlein$)  is the class of Eberlein compacta (uniform Eberlein compacta, respectively)
  while $\ccc$ stands for all compact $ccc$ spaces.
Moreover, we denote by  $\ale$  the class of `singleton-supported compact spaces',
i.e.~those compact spaces which embed into $\{x\in \erre^\kappa: x(\gamma)\neq0 \mbox{ for at most one }\gamma<\kappa\}$ for some $\kappa$. Clearly, every compact space in $\ale$ is Uniformly Eberlein (recall Farmaki's characterization of uniformly Eberlein compacta in \cite[Theorem 2.10]{Farmaki}).

\begin{remark}
\label{RemClassesWhoseOrthogonalIsCCC}
The following chain of implications holds
$$ \singletonsupported \subset \ueberlein \sub \eberlein \sub \Talagrand \sub \Gulko. $$
\end{remark}

Indeed, Talagrand \cite{Ta75} proved that every Eberlein compact space is
Talagrand compact and, motivated by Rosenthal's Theorem mentioned above, he asked whether every
Talagrand or Gul'ko $ccc$ compact space is metrizable
(see \cite[Probl\`eme 5]{Ta77} and \cite[Probl\`eme 7.9]{Ta79}).
These questions were solved positively by Argyros and Negrepontis in \cite{AN83}.
The following may be seen as a summary of these results.

\begin{proposition}\label{occ:2}
	Let $\cC$ be any of the classes in Remark \ref{RemClassesWhoseOrthogonalIsCCC}.
Then $\cC^\perp = \ccc$. As a consequence,  $\cC ^{(\perp)}$ coincides with the class of hereditarily ccc compact spaces.
\end{proposition}

\begin{proof}
As $\singletonsupported \subset \Gulko$, we have $ \Gulko^\perp \subseteq \singletonsupported^\perp $.
Clearly, the class $\ccc$ is closed under continuous images.
So if a continuous image $L$ of $K\in\ccc$ is Gul'ko then $L$ is metrizable by the result from \cite{AN83}.
Thus,  $\ccc \subseteq  \Gulko^\perp$.

On the other hand, if $K$ is not $ccc$ and $\{W_\alpha : \alpha<\omega_1\}$
are pairwise disjoint nonempty open sets,  one can take for every $\alpha$ a
nonzero continuous function $f_\alpha:K\To [0,1]$  such that $f_\alpha|_{K\setminus W_\alpha} = 0$.
Then the diagonal mapping maps $K$ onto a nonmetrizable subset of
\[\{x\in \erre^{\omega_1}: x(\gamma)\neq0 \mbox{ for at most one }\gamma<\omega_1\},\]
 so
	$K\notin \singletonsupported^\perp $.
This proves that $\ccc =  \singletonsupported^\perp  =\Gulko^\perp$ and the first part of the proposition follows from Remark \ref{RemClassesWhoseOrthogonalIsCCC}.

The second assertion follows from Lemma \ref{lemhereditarilyCperp}.
\end{proof}

\begin{corollary}\label{occ:3}
	$\ccc=\ccc^{\perp\perp}$.
\end{corollary}

\begin{proof}
	Applying  Lemma \ref{occ:1} and Proposition \ref{occ:2}   \[\ccc=\eberlein^\perp=\eberlein^{\perp\perp\perp}=\ccc^{\perp\perp},\]
and we are done.
\end{proof}

A natural class generalizing all the classes mentioned in Remark
\ref{RemClassesWhoseOrthogonalIsCCC} is the class of Corson compacta, i.e.~those compact spaces $K$ that are homeomorphic to a subspace of a $\Sigma$-product
\[\Sigma(\erre^\kappa):=\{x\in \erre^\kappa: x \mbox{ has countable support}\},\]
for some cardinal $\kappa$.
As we will recall below, it is consistent that there are $ccc$ Corson compact spaces which are not metrizable;
in such a case  $\corson^\perp$ is a proper subclass of  $\ccc$.

In the  spirit of Talagrand's question mentioned above,
one may wonder what is the largest class whose orthogonal is $\ccc$.
Note that Corollary \ref{occ:3} provides {\em some} answer ---
the largest class whose orthogonal is $\ccc$ is just $\ccc^\perp$. Nevertheless, no handy characterization of $\ccc^\perp$ in terms of classical classes of compact spaces
seems to be available.

\begin{proposition}\label{occ:4}
	Under Marin's axiom MA$_{\omega_1}$, $\ccc^\perp = \separable^\perp$.
\end{proposition}

\begin{proof}
	Clearly, $\ccc^\perp \subseteq \separable^\perp$. If $K\notin \ccc^\perp$ then
$K$ has a nonmetrizable $ccc$ continuous image $L$.
	It is easy to construct  a continuous image $L_1$ of $L$ of weight
$\omega_1$. Then $L_1$ is also  $ccc$ and the standard application of Martin's axiom
implies that $L_1$  is separable. Consequently,  $K\notin\separable^\perp$. 	
\end{proof}

Recall that  $\omega_1$ is  a precaliber of a topological space $X$ if every
uncountable family of nonempty open subsets of $X$ contains an uncountable centered subfamily.
If we take a compact space $K$ then the notion of a precaliber coincides with that of a caliber:
$\omega_1$ is a caliber of $K$ if
for every uncountable family $\cG$ of  nonempty  open  subsets of $K$,
$|\{G\in\cG: x\in G\}|\ge \omega_1$ for some $x\in K$
(cf.\ \cite{CN82}). Write $\caloo$ for the class of compact spaces having caliber $\omega_1$.

\begin{remark}\label{occ:5}
	$\caloo=\ccc$ is equivalent to MA$_{\omega_1}$. Indeed, the condition $K\in\caloo$ is equivalent to saying that every point-countable
	family of open subsets of $K$ is countable; the latter is called Shanin's condition, see Todorcevic \cite{To00}.
\end{remark}

\begin{proposition}\label{occ:6}
	$ \corson^\perp=\caloo$. Thus,  $\corson^\perp=\ccc$ is equivalent to  MA$_{\omega_1}$.
\end{proposition}

\begin{proof}
	Note first that if $L$ is nonmetrizable Corson compact then $\omega_1$ is not a caliber of $L$. Indeed,
	suppose that $L$ is embedded into the $\Sigma$-product $\Sigma(\erre^\kappa)$ for some $\kappa$.
	Then the set $V_\alpha=\{x\in K: x(\alpha)\neq 0 \}$ is nonempty for uncountably many $\alpha$'s.
	On the other hand, no uncountable subfamily of $V_\alpha$'s can have nonempty intersection by the very
 definition of  a $\Sigma$-product.
	
	If $\omega_1$ is a caliber of a compact space $K$ then it is also a caliber of every continuous image $L$ of $K$.
	Hence such $L$ is either metrizable or not Corson by above.
	
	If $\omega_1$ is not  a caliber of $K$ then take a family $\{G_\alpha:\alpha<\omega_1\}$ witnessing that fact.
	We can of course assume that $G_\alpha\neq K$ for every $\alpha$. If $g_\alpha: K\to [0,1]$ is a nonzero continuous function vanishing outside $G_\alpha$ then the diagonal mapping $\Delta_{\alpha} g_\alpha:K\To[0,1]^{\omega_1}$ maps
	$K$ onto  a nonmetrizable Corson compactum.
\end{proof}

\begin{remark}\label{occ:6.5}
	Arguing as above we can check that the failure of MA$_{\omega_1}$, that is $\caloo\neq\ccc$, is equivalent to the existence of
	zerodimensional nonmetrizable $ccc$ Corson compacta.
\end{remark}

We recall  below that Proposition \ref{occ:4} does not hold in ZFC. We have three natural classes of compacta
\[\separable\subseteq \SPM\subseteq \ccc,\]
where $\SPM$ is the class of compact spaces carrying a strictly positive probability Borel measure.
Therefore,
\[\ccc^{\perp} \subseteq \SPM^{\perp} \subseteq \separable^{\perp}.\]

Obviously, every $\aleph_0$-monolithic space (i.e.\
a space in which all separable subspaces are metrizable),
in particular every Corson compactum,  is in $\separable^\perp$.
There is a large number of examples, starting from those  constructed by
Haydon, Kunen and Talagrand under
CH and described in \cite{Ne84}, of a nonseparable
Corson compact space $K\in\SPM$. They indicate that the class $\SPM\cap \separable^\perp$ is nontrivial
(i.e.\ contains nonmetrizable spaces) under CH.
Kunen and van Mill \cite{KM95} showed that, in this context,  CH may be relaxed
 to saying that Martin's axiom fails for measure algebras.

Another result showing that Proposition \ref{occ:4} may be violated is based on classical Gaifman's example of
a $ccc$ compact space admitting no strictly
positive measure. The proof below follows closely the description of the Gaifman space given by Comfort and Negrepontis
\cite[Theorem 6.23]{CN82} and its modification from \cite{AMN}.

By a Lusin set we mean an uncountable $Z\sub\er$ such that $Z\cap N$ is countable for every nowhere dense $N\sub\er$.
Recall that a family $\cA$ of subsets of some set $X$ is said to be adequate
if $B\sub A\in\cA $ always implies $B\in\cA$ and $A\in\cA$ for every $A\sub X$
satisfying $[A]^{<\omega}\sub\cA$.  Those two properties imply that
an adequate family $\cA$ defines a compact subset of $2^X$ in a natural way.

Below we write $P(K)$ for the space of regular probability Borel measures on a given compact space.
Recall that $K\in\SPM$ is equivalent to saying that $K$ is the support of some regular measure $\mu\in P(K)$ (see
\cite{CN82}). Recall also that if $g:K\to L$ is a continuous surjection of compacta
	 and $\nu\in P(L)$ then there is $\mu\in P(K)$ such that $g[\mu]=\nu$
	(meaning that $\nu(B)=\mu(g^{-1}[B])$ for every Borel set $B\sub L$).

\begin{proposition}\label{occ:7}
	If there is a Lusin set then there is a nonseparable Corson compact space in $\ccc\cap \SPM^\perp$.
\end{proposition}

\begin{proof}
	Fix a Lusin set $Z$.
Let $\{T_n: n\ge 2\}$ be an enumeration of intervals in $\er$ with rational endpoints.
	For every $n$ we pick a pairwise disjoint family $\{T_{nk}: 1\le k\le n^2\}$ of nonempty subintervals of $T_n$.
	
	We define an adequate family $\cA$ of subsets of $Z$ as follows:
	\[A\in\cA\iff (\forall n)\; |\{k\le n^2: A\cap T_{nk}\neq\emptyset\}|\le n.\]
	Then we consider a compact space $K\sub 2^Z$ defined by $\cA$, that is
	\[ \chi_A\in K \iff A\in\cA.\]
	The space $K$ is $ccc$; in fact it is proved in \cite{CN82} that the family of all nonempty open subsets of
$K$ may be written
	as $\bigcup_n  \cG_n$ where every $\cG_n$ contains no pairwise disjoint subfamily of size $n$.
	
	The space $K$ is Corson compact since $\cA$ contains no uncountable set ---
note that  if $X\sub Z$ is uncountable
	then $X$ is not meager in $\er$ so it is dense in some $T_n$ and intersects every $T_{nk}$; therefore $X\notin\cA$.
	\medskip
	
	\noindent {\sc Claim.}
	Every $\mu\in P(K)$ has a metrizable support.
	\medskip
	
	Indeed, fix such a measure $\mu$ and write $V_z=\{x\in K: x(z)=1\}$.
If we suppose that $\mu$ is not concentrated on a metrizable subspace of $K$ then $\mu(V_z)>0$ for uncountably many $z\in Z$. Then there is $\eps>0$ and an uncountable $Z_0\sub Z$ such that
	$\mu(V_z)\ge\eps$ for $z\in Z_0$. By the Lusin property $Z_0$ is not meager in $\er$
so its closure contains  $T_n$ for infinitely many $n$; fix such $n$ with $n>1/\eps$.
	Pick $z_k\in T_{nk}\cap Z_0$ for every $k\le n^2$ and consider the  family of sets $V_{z_k}$, $k\le n^2$ each of measure $\ge\eps$.
	It follows that there is $I\sub \{1,2,\ldots, n^2\}$ with $|I|\ge \eps\cdot n^2$ and such that
	$W=\bigcap_{i\in I} V_{z_i}\neq\emptyset$. To see this note that, otherwise, we would have \[ \sum_{i=1}^{n^2}\chi_{V_{z_i}}<\varepsilon\cdot n^2, \mbox{ and }
\varepsilon\cdot n^2 \leq \int_K \sum_{i=1}^{n^2}\chi_{V_{z_i}}\;{\rm d}\mu < \varepsilon\cdot n^2.\]
 On the other hand $W=\emptyset$ by the very definition of $\cA$, as $\eps\cdot n^2> n$, which is a contradiction.
 \medskip

	Using Claim we can check that $K\in \SPM^\perp$: let $g:K\to L$ be a continuous surjection.
	If $L$ is nonmetrizable and $\nu\in P(L)$ then there is $\mu\in P(K)$ such that $g[\mu]=\nu$.
By Claim,  $\mu(K_0)=1$ for some metrizable $K_0\sub K$;
	consequently, $\nu$ lives on a metrizable subspace $g[K_0]$ of $L$ and hence $\nu$ cannot be strictly positive.
\end{proof}
\bigskip

In order to provide more examples of spaces in $\ccc^\perp$, we use the following simple lemma.

\begin{lemma}
	\label{auxlemmaclasses}
	Let $\cC_1$, $\cC_2$ and $\cC_3$ be three classes of compact spaces with $\cC_2$ stable under continuous images and such that $\cC_1\cap \cC_2 \subseteq \cC_3$. Then, $\cC_3 ^\perp \cap \cC_2 \subseteq \cC_1^\perp$.
\end{lemma}

\begin{proof}
	Suppose that  $K \in \cC_3 ^\perp \cap \cC_2$ and take a continuous image $L$ of $K$ which is in $\cC_1$. 
Then, since $\cC_3 ^\perp$ and $\cC_2$ are stable under continuous images, we have $L \in \cC_3 ^\perp \cap \cC_2 \cap \cC_1 \subseteq  \cC_3 ^\perp \cap  \cC_3 $, so $L$ is metrizable.
\end{proof}

Let $\cC_{dm}$ be the class of  compacta $K$ having the property that  
 every continuous image of $K$  has a dense metrizable subspace. 
 Notice that this class is stable under continuous images. It contains the class
of fragmentable compact spaces and the (consistently) more general class of Stegall compact spaces (see, c.f., \cite[Theorems 3.1.5, 3.1.6 and 5.1.11]{FabianGD}).
%

\begin{lemma}
	\label{auxlemmacccdm}
	Any $ccc$ compact space in $\cC_{dm}$ is separable.
\end{lemma}
\begin{proof}
	Take $K\in\ccc\cap \cC_{dm}$; since $K$ is a continuous image of itself, there is a metrizable dense subset $D\subseteq K$. As dense subspaces of $ccc$ spaces are $ccc$, we conclude that $D$ is a $ccc$ metrizable topological space, so it is separable and, therefore, $K$ is also separable.
\end{proof}

\begin{proposition} \label{occ:8}
	$\separable ^ \perp \cap \cC_{dm} \subseteq \ccc ^\perp$. In particular, any $\aleph_0$-monolithic fragmentable compactum $K$ is in $\ccc^\perp$.
\end{proposition}

\begin{proof}
Lemma \ref{auxlemmacccdm}  says that we can apply  Lemma \ref{auxlemmaclasses} to $\cC_1=\ccc$,
$\cC_2=	\cC_{dm}$ and $\cC_3=\separable$ and this yields 
$\separable ^ \perp \cap \cC_{dm} \subseteq \ccc ^\perp$.
	
The last part of the proposition follows from the fact that any $\aleph_0$-monolithic compact space is in $\separable ^\perp$.
\end{proof}

Ordinal interval spaces $[0,\alpha]$ serve as examples of compact spaces $K$ as above; 
they are scattered and therefore fragmentable.

We finish this section with a remark: Every nonmetrizable continuous image of a compact space $K$ has a further continuous image of weight $\omega_1$. Therefore, if a class $\cC$ is stable under continuous images, then a compact space  $K\in \mathcal{C}^\perp$ if and only if no continuous image of $K$ of weight $\omega_1$ belongs to $\mathcal{C}$.  In this sense, there is some similarity between the questions we are considering here and reflection problems of the sort studied by Tkachuk \cite{Tk12} and Tkachenko and  Tkachuk \cite{TT15}:
Given a class of compacta $\cC$, can we say that $K\in\cC$ once we know that
every continuous image of $K$ of weight $\le \omega_1$
is in $\cC$? Answering two questions from \cite{TT15}, Magidor and Plebanek \cite{MP17} gave a consistent example of a scattered non Corson compact space all of whose continuous images of weight $\omega_1$ are uniform Eberlein. Such a space would be a kind of extreme example in the class $\ccc^\perp$.
%
%
%
%
%
%
%

\section{The orthogonal class of (weakly) Radon-Nikod\'ym compacta}
\label{sectionOrthogonalWRN}

A compact space is Radon-Nikod\'ym  if it is homeomorphic to a $weak^\ast$ compact subset
of $X^\ast$ for some Banach space $X$ which is Asplund, that
is every separable  subspace $Y$ of $X$ has a separable dual.
Note that if $X$ is Asplund then it does not contain an isomorphic copy of $\ell_1$;
hence $\wrn\sub\rn$. The reverse inclusion does not hold in the following strong sense.
%

\begin{lemma}
	\label{LemmaSplitInterval}
	Every hereditarily Lindel\"{o}f compact space is in $\rn ^\perp$. 
In particular, the class $\wrn \cap \rn ^\perp$ is nontrivial, i.e.\ it contains nonmetrizable
compacta.
\end{lemma}

\begin{proof}
Let $K$ be an hereditarily Lindel\"{o}f compact space and suppose that 
$L$ is a Radon-Nikod\'ym continuous image of $K$. 
Since the Lindel\"{o}f property is stable under continuous images, $L$ is hereditarily Lindel\"{o}f and therefore metrizable by \cite[Theorem 5.8]{Nam87}. Thus, $K$ belongs to $\rn ^\perp$.

The second statement of the lemma follows from the fact that there 
exist nonmetrizable hereditarily Lindel\"{o}f compact spaces which are WRN. 
One example is given by the \textit{double arrow space}, also known as the \textit{split interval} 
(see \cite[Example 5.9]{Nam87} and \cite[Corollary 8.8]{GlasnerMegrelishvili14}).
\end{proof}
%
%
%
%
%

Since $\eberlein\sub\rn\sub\wrn$, we have
\[\wrn^\perp \sub \rn^\perp\sub\eberlein^\perp=\ccc.\]

By Lemma \ref{LemmaSplitInterval}, the set $\wrn \cap \rn ^\perp$ is nontrivial. Let us observe
that $\wrn^\perp\neq\ccc$:
for instance, 
$\beta\mathbb{N}\in\ccc\sm \wrn^\perp$  since $\beta\mathbb{N}$
can be continuously mapped onto any separable compact space.

Dyadic spaces provide natural examples of compact spaces that are in 
 $\wrn^\perp$.  Recall that a compact space is said to be
dyadic if it is a continuous image of the Cantor cube $2^\kappa$ for some  $\kappa$.

\begin{proposition}\label{dyadic}
	Every dyadic compactum belongs to $\wrn^\perp$.
\end{proposition}

\begin{proof}
	If $K$ is dyadic then, by a result due to Gerlits \cite{Ge76}, every nonmetrizable continuous image $L$  of $K$
	contains a copy of $2^{\omega_1}$, so $L$ is not in $\wrn$ by \cite[Remark 10.7]{GlasnerMegrelishvili14}.
\end{proof}

	As mentioned in the introduction, following \cite{AMP18} we say  that a  Boolean algebras $\fA$ is {\em weakly Radon-Nikod\'{y}m} (denoted $\fA\in\wrnb$)  if its Stone space is in the class $\wrn$.
Theorem \ref{intro:wrn} gives an internal characterizations of such algebras.
It is convenient to rephrase that purely Boolean condition  in the following form
(see \cite[Proposition 3.2]{AMP18} for the proof).

\begin{theorem}\label{ba:1}
For a Boolean algebra $\fB$, $\fB\in\wrnb$ if and only if $\fB$ can be decomposed into countably many parts,
none of which contains an infinite independent sequence.
\end{theorem}

Recall that, for a class $\cC$ of Boolean algebras, we denote by $\cC^\perp$ the class of Boolean algebras containing no uncountable Boolean subalgebra in $\cC$. Theorem \ref{ba:1} yields the following  characterization of $\wrnb^\perp$.

	\begin{proposition}\label{ba:2}
Given a Boolean algebra $\fB$, $\fB \in \wrnb^\perp$ if and only if every uncountable subset of
$\fB $ contains an infinite independent subset.
	\end{proposition}

	\begin{proof}
Consider an uncountable set $\Gamma \sub \fB$ where $\fB \in \wrnb^\perp$.
Let $L$ be the Stone space of the Boolean subalgebra $\fA$
		of $\fB$  generated by~$\Gamma$; then $L$ is a continuous image of $\ult(\fB)$.
Since $\Gamma$ is uncountable, $L$ is not
	metrizable and so it is not weakly Radon-Nikod\'ym. In particular,
this and Theorem \ref{ba:1} immediately imply that $\Gamma$ contains an infinite independent sequence.
		
		For the reverse implication suppose that
		 $\fA$ is a subalgebra of $\fB$ such  that $\fA\in\wrnb$. Then $\fA$ is
		generated by some $\cG=\bigcup_n\cG_n$, where no $\cG_n$
		contains an infinite independent sequence. It follows that every $\cG_n$ is countable so $\fA$ is countable as well.
	\end{proof}

	\begin{remark}\label{ba:3}
		We can compare the above characterization of Boolean algebras in $\wrnb^\perp$ with the following:
		$\ult(\fB)\in\scattered^\perp$ if and only if every uncountable {\bf subalgebra} of $\fB$ contains an infinite
		independent sequence. This follows from the fact that, given $\fA\sub\fB$,
		$\ult(\fA)$ is not scattered if and only if $\ult(\fA)$ maps continuously onto the Cantor set $2^\omega$.
	\end{remark}
	
\section{Corson compacta and $\wrnb$}
\label{SectionCorsonandWRNB}

As we have seen, all dyadic compacta are in the class  $\wrn^\perp$.
To find other examples of nonmetrizable spaces from $\wrn^\perp$
we need some auxiliary results.

Consider any subspace $X$ of some cube $2^\kappa$.
We say that some $C\sub X$ is determined by (coordinates in) $J\sub \kappa$ if
\[ \pi^{-1}_J \pi_J(C)\cap X=C.\]
In other words: if $x\in C$, $y\in X$, $x|J=y|J$ then $y\in C$. Notice that the ambient space $X$ is relevant in this definition, even if we do not mention it for economy of language.
Recall that if $K\sub 2^\kappa$ is compact then, by the Stone-Weierstrass theorem, every clopen subset of $K$ is
determined by a finite number of coordinates.

Let us first fix $n=\{ 0,1\ldots, n-1\}$ and some $X\sub 2^n$; for any $s\sub n$ and $\vf :s\to 2$ write
\[ C(\vf)=\{x\in X: x|s=\vf \}.\]

\begin{lemma}\label{clopen:1}
Suppose that $J\sub n$ and $C\sub X$ is a subset that is {\bf not} determined by coordinates in $J$ (this, in particular,
implies that $C$ is a proper subset of $X$).
Then there are $s\sub n$,  $k\in s\sm J$ and $\vf,\psi:s\to 2$ such that

\begin{enumerate}[(i)]
\item $\vf|(s\sm\{k\})=\psi|(s\sm\{k\})$;
\item $\emptyset\neq C(\vf)\sub C$;
\item $\emptyset\neq C(\psi)\sub X\sm C$.
\end{enumerate}
\end{lemma}

\begin{proof}
Since $C$ is not determined by $J$, there are $x\in C$ and $y\in X\sm C$ such that $x|J=y|J$.
Choose such a pair $x,y$  that the set $\Delta(x,y)=\{i<n: x(i)\neq y(i)\}$ has the minimal possible size.
As $x,y$ are different, the set $\Delta(x,y)$ is not empty; choose any $k\in \Delta(x,y)$; note that $k\notin  J$.
Put $s=(n\sm \Delta(x,y))\cup\{k\}$ and define $\vf,\psi:s\to 2$ so that $\phi|s=x|s$ and $\psi|s=y|s$.
Then $(i)$ is granted and we have $C(\vf), C(\psi)\neq\emptyset$, so it remains to verify the inclusions
in $(ii)$ and $(iii)$.

Suppose that $z\in C(\vf)$ but $z\in X\sm C$. Then $\Delta(x,z)\sub \Delta(x,y)\sm \{k\}$, a contradiction with
the minimality of $\Delta(x,y)$.

We verify $(iii)$ in a similar manner: Suppose that $z\in C(\psi)$ but $z\in C$.
Then again $\Delta(z,y)$ is a proper subset of $\Delta(x,y)$.
\end{proof}

Consider now a compact space $K\sub 2^\kappa$ for some $\kappa$. As before we write
\[ C(\vf)=\{x\in K: x|s=\vf\},\]
for any finite $s\sub\kappa$ and $\vf:s\to 2$.

\begin{corollary}\label{clopen:2}
Suppose that $\cC$ is an uncountable family of clopen subsets of $K$. Then there
are families $\{C_\alpha:\alpha<\omega_1\}\sub \cC$,
$\{s_\alpha:\alpha<\omega_1\}\sub [\kappa]^{<\omega}$,
a one-to-one function $\xi:\omega_1\to\kappa$ and $\vf_\alpha,\psi_\alpha : s_\alpha \to 2$ such that
for every $\alpha<\omega_1$

\begin{enumerate}[(i)]
\item $\vf_\alpha |(s_\alpha\sm\{\xi(\alpha)\})=\psi_\alpha |(s_\alpha \sm\{\xi(\alpha)\})$;
\item $\emptyset\neq C(\vf_\alpha)\sub C_\alpha$;
\item $\emptyset\neq C(\psi_\alpha)\sub K\sm C_\alpha$.
\end{enumerate}
\end{corollary}

\begin{proof}
This follows from the lemma above by a simple induction: note that for every countable
$J\sub\kappa$ there are only countably many clopens that are determined by coordinates in $J$.
\end{proof}

We shall now consider any $ccc$ compact space $K\sub 2^\kappa$ and its `adequate closure' $\widetilde{K}$,
where
\[ \widetilde{K} = \left\{ x\in 2^\kappa :   \mbox{ there is } y\in K \mbox{ such that }  x \leq y \right\}.\]
In what follows, we write $f \leq g$ for functions defined on possibly different subsets of $\kappa$ if $f(\alpha) \leq g(\alpha)$ whenever $\alpha$ belongs simultaneously to the domains of $f$ and $g$. Recall that we say that a zero-dimensional compact space belongs to  $\wrnb^\perp$ if its clopen algebra belongs to $\wrnb^\perp$. As we shall see in the next section, compact spaces in $\wrnb^\perp$ might not belong to $\wrn^\perp$.

\begin{theorem}\label{cs:1}
	If a compact space $K \subset 2^\kappa$ is $ccc$
then its adequate closure $\widetilde{K}$ belongs to $\wrnb^\perp$.
\end{theorem}

\begin{proof}
In view of Proposition \ref{ba:2}, we are going to show that every uncountable subfamily of
$\clop(\widetilde{K})$ contains an infinite independent sequence.

Using Corollary \ref{clopen:2} it is enough to consider
$\{C_\alpha:\alpha<\omega_1\}\sub \clop(\widetilde{K})$ together with
$\{s_\alpha:\alpha<\omega_1\}\sub [\kappa]^{<\omega}$,
a one-to-one function $\xi:\omega_1\to\kappa$ and $\vf_\alpha,\psi_\alpha:s_\alpha  \to 2$ such that
for every $\alpha<\omega_1$

\begin{enumerate}[(i)]
\item $\vf_\alpha |(s_\alpha\sm\{\xi(\alpha)\})=\psi_\alpha |(s_\alpha \sm\{\xi(\alpha)\})$;
\item $\emptyset\neq C(\vf_\alpha)\sub C_\alpha$;
\item $\emptyset\neq C(\psi_\alpha)\sub \widetilde{K}\sm C_\alpha$.
\end{enumerate}

Now we show that there are infinitely many independent pairs
of the form $(C(\vf_\alpha), C(\psi_\alpha))$. Set $A_\alpha=C(\vf_\alpha)$ whenever
$\vf_\alpha (\xi(\alpha))=1$ and $A_\alpha=C(\psi_\alpha)$ otherwise, that is
when $\psi_\alpha (\xi(\alpha))=1$.
By definition of $\widetilde{K}$, for every $x\in A_\alpha$ there is $y_x \in K$ such that $x\leq y_x$. Thus, we can take $\{s_\alpha':\alpha<\omega_1\}\sub [\kappa]^{<\omega}$ with $s_\alpha \subseteq s_\alpha'$ and $g_\alpha : s_\alpha' \rightarrow 2$ such that $g_\alpha \geq \max\{\psi_\alpha, \vf_\alpha\}$, and $C(g_\alpha) \cap K \neq \emptyset$ for every $\alpha<\omega_1$.

Now we apply $ccc$ to the family $\{C(g_\alpha) \cap K:\alpha<\omega_1\}$
of clopens to get $x\in K$ belonging to $C(g_{\alpha_n})$ for  a sequence of distinct $\alpha_n$'s. It follows from the inequality $x \geq \max\{\psi_{\alpha_n} , \vf_{\alpha_n}\}$ for every $n\in \N$ and the definition of $\widetilde{K}$ that the pairs $(C(\vf_{\alpha_n}), C(\psi_{\alpha_n}))$
are independent.
\end{proof}

\begin{corollary}\label{cs:2}
	If MA$_{\omega_1}$ does not hold
	then  there is a nonmetrizable zero-dimensional compact space in $\corson \cap \wrnb^\perp$.
\end{corollary}

\begin{proof}
	We apply Theorem \ref{cs:1} to a zerodimensional nonmetrizable $ccc$  Corson compactum $K$ (see Remark \ref{occ:6.5}).
	Notice that $\widetilde{K}$ is also Corson.
\end{proof}

A Boolean algebra $\fB$ has a precaliber $(\kappa,\lambda)$ if every family in $\fB$ of size
$\kappa$ contains a centered
subfamily of size $\lambda$.
Mimicking this definition (it may refer either to
Boolean algebras or topological spaces) we can form the following.
Say that a Boolean algebra $\fB$ has an independence-precaliber $(\kappa,\lambda)$ if
every subfamily of $\fB$ of size $\kappa$ contains an independent subfamily of size $\lambda$.
Such a  notion was already considered in the context of measure algebras, see  \cite{DP04}
 and also \cite{FP04}.
 With  this terminology, Proposition \ref{ba:2} says that $\fB\in\wrnb^\perp$ if and only if
$\fB$ has an independence-precaliber $(\omega_1,\omega)$. Arguing as in Corollary \ref{cs:2} we conclude the following.

\begin{corollary}\label{cs:3}
	Suppose that MA$_{\omega_1}$ does not hold.
	Then there is an uncountable  Boolean algebra  $\fB$ such that
	
	\begin{enumerate}[(i)]
		\item $\fB$ has an independence-precaliber $(\omega_1,\omega)$,  i.e.\ $\fB\in\wrnb^\perp$;
		\item $\fB$ does not have a precaliber $(\omega_1,\omega_1)$; in fact, $\fB$ 
 contains no uncountable independent family.
	\end{enumerate}
\end{corollary}

\begin{remark}
We enclose two comments on Theorem \ref{cs:1}.

\begin{enumerate}
\item From the purely  algebraic point of view,
Theorem \ref{cs:1}  says that if $\{a_i : i\in I\}$ is a family of elements
in a $ccc$ Boolean algebra, and $\{e_i : i\in I\}$ are independent,
then the algebra generated by $\{a_i\otimes e_i : i\in I \}$ in the free product
belongs to $\wrnb^\perp$.
\item The proof of Theorem  \ref{cs:1} shows that any chain condition on $K$ gives the analogous independence condition on $\tilde{K}$. For example, if $K$ has a  precaliber $\omega_1$ (that is, $K\in \corson^\perp$), then $\clop(\widetilde{K})$
has an independence-precaliber $(\omega_1,\omega_1)$.
\end{enumerate}
\end{remark}

\section{Between $\wrnb$ and $\wrn$} \label{between}

 The main result from the previous section leads us to the following question.

\begin{problem}\label{between:1}
Is there a (consistent) example of a nonmetrizable Corson compact space in $\wrn^\perp$?
\end{problem}

In particular, we do not know if the space $\widetilde{K}$ discussed in Theorem \ref{cs:1} is orthogonal
to all (not necessarily zero-dimensional) weakly Radon-Nikod\'ym compacta.

In connection to Problem \ref{between:1} we shall now present another construction
showing that a compact space in $\wrnb^\perp$ might not belong to $\wrn^\perp$.

\begin{theorem}\label{between:2}
Let $L$ be a $ccc$ compact and convex subspace of $\er^\kappa$ (for some $\kappa)$.
Then $L$ is a continuous image of a space $K$ belonging to $\wrnb ^\perp$.
\end{theorem}

\begin{proof}
With $L\sub \er^\kappa$ given, we first fix some notation. For any $a\in \er$ and $\xi<\kappa$
write
\[ V_\xi^0(a)=\{x\in L : x(\xi)<a\},\quad V_\xi^1(a)=\{x\in L : x(\xi)>a\}.\]
Moreover, we denote
$ \Delta=\{\la p,q\ra\in \qu^2: p<q\}.$

We shall consider functions $f:\kappa\times\Delta\to 2$.
For such $f$ and a finite set $s\sub \kappa\times\Delta$ we set
$$ \core_s(f)=\left\{x\in L : \forall \la \xi,p,q\ra \in s\ (f(\xi,p,q)=0 \Rightarrow x_\xi <q), \ (f(\xi,p,q)=1 \Rightarrow x_\xi >p) \right\}=$$\[ =\bigcap \left\{V_\xi^0(q): \exists p,~ \la \xi,p,q\ra\in s \mbox{ and } f(\xi,p,q)=0\right\}\cap\]
\[\cap \bigcap\left\{V_\xi^1(p): \exists q,~ \la \xi,p,q\ra\in s \mbox{ and } f(\xi,p,q)=1\right\}.\]
Note that $\core_s(f)$ is in fact determined by $f|s$ so below we also
consider $\core_s(\vf)$ whenever some $\vf:s\to 2$ is given.

We define the space $K$ as follows
\[K=\left\{f:  \core_s(f)\neq\emptyset\mbox{ for every finite } s\sub \kappa\times\Delta\right\},\]
and check that $L$ is a continuous image of $K$ and that $K\in\wrnb^\perp$.
By the very definition, $K$ is a compact subspace of $2^{\kappa\times\Delta}$.
\medskip

\noindent {\sc Claim 1}. There is a continuous surjection $\theta: K\To L$.
\medskip

Define $\theta(f)$ to be the unique point in $\bigcap_s \overline{ \core_s(f)}$ (where the intersection is taken over
all finite $s\sub\kappa\times\Delta$). To see that the definition is correct note that such an intersection is
nonempty by compactness of $L$. Moreover, $\bigcap_s \overline{ \core_s(f)}$ cannot contain two
distinct points $x,x'\in L$ for, otherwise, we have $x(\xi)\neq x'(\xi)$ for some $\xi<\kappa$;
say that   $x(\xi)<  x'(\xi)$ and we can pick rational numbers $p,q$ so that $x(\xi)<p<q<x'(\xi)$.
Then
examine the value of $f(\xi,p,q)$ to get a contradiction.
Moreover, if $x \in L$, then it is immediate that the function $f:\kappa \times \Delta \rightarrow 2$ given by $f(\xi, p,q)=1$ if $x(\xi)>p$ and zero otherwise belongs to $K$ and satisfies $\theta(f)=x$, so $\theta$ is surjective. 

To verify the continuity of $\theta$ note that sets of the form $V_\xi^1(p),~V_\xi^0(q)$ form a subbase of the topology on $L$. The following equalities show that the preimages of these sets under $\theta$ are open:

\[ \theta^{-1}(V_\xi^1(p))=\bigcup_{p',q\in\qu, ~p<p'<q}\{f \in K: f(\xi,p',q)=1\},\]
\[\theta^{-1}(V_\xi^0(q))=\bigcup_{p,q'\in \qu,~p<q'<q}\{f \in K: f(\xi,p,q')=0\}.\]
We carefully check the first one, the second one being analogous. If $f(\xi,p',q)=1$ for some $q>p'>p$ then $\core_s(f) \subset V_\xi^1(p')$ for any $s$ that contains $(\xi,p',q)$. So
\[\theta(f) \in \bigcap_s \overline{\core_s(f)} \subseteq \overline{V_\xi^1(p')} \subset V_\xi^1(p).\]
For the reverse inclusion, suppose that $f(\xi,p',q)=0$ whenever $p<p'<q$. Then $\core_s(f) \subset V_\xi^0(q)$ whenever $(\xi,p',q)\in s$. Therefore
\[\theta(f)\in \bigcap_s \overline{\core_s(f)} \subset \bigcap_{q>p}\overline{V_\xi^0(q)}.\]
It follows that $\theta(f)_\xi \leq p$, so that $\theta(f)\not\in V_\xi^1(p)$.

\bigskip
It remains to prove that $K\in\wrnb ^\perp$. Consider first a finite set
$s\sub \kappa\times\Delta$ and two functions $\vf,\psi:s\to 2$ that
differ only at $\la\xi_0,p_0,q_0\ra\in s$; say that $\vf(\xi_0,p_0,q_0)=0$ and $\psi(\xi_0,p_0,q_0)=1$.
Assume that the clopens $C(\vf)$ and $C(\psi)$ are nonempty.

Set $s'=s\sm\{\la \xi_0,p_0,q_0\ra\}$; note that $\core_{s'}(\vf)=\core_{s'}(\psi)$ and
\[ \core_{s}(\vf)=\core_{s'}(\vf)\cap V_{\xi_0}^0(q_0),\quad \core_{s}(\psi)=\core_{s'}(\psi)\cap V_{\xi_0}^1(p_0).\]

\noindent {\sc Claim 2}.
The set $U= \core_{s'}(\vf)\cap V_{\xi_0}^0(q_0)\cap V_{\xi_0}^1(p_0) = \core_s(\varphi)\cap \core_s(\psi)$ is not empty.
\medskip

This follows from the convexity of $L$: take $x\in \core_s(\vf)$ and $y\in \core_s(\psi)$.
Then $x(\xi_0)<q_0$ and $y(\xi_0)>p_0$ so there is $z$ lying on the segment joining $x$ and $y$
such that $z(\xi_0)\in (p_0,q_0)$. Hence $z\in U$ (note that every $\core$ is a convex subset of $L$).
\medskip

In order to prove that $K\in\wrnb^\perp$ we check the criterion of Proposition \ref{ba:2} as we did before:
if $\cC\sub\clop(K)$ is an uncountable
family then, as in Theorem \ref{cs:1}, we find uncountably many $C_\alpha\in\cC$ for which
there are nonempty sets $C(\vf_\alpha)\sub C_\alpha$ and $C(\psi_\alpha)\sub K\sm C_\alpha$,
where $\vf_\alpha$ and $\psi_\alpha$ have the same finite domain $s_\alpha = s'_\alpha\cup\{\zeta_\alpha\}$ and differ at exactly one point $\zeta_\alpha$. As stated in Corollary~\ref{clopen:2}, the assignment $\alpha\mapsto \zeta_\alpha$ can be taken one-to-one. We can also suppose that the sets $s'_\alpha$ form a $\Delta$-system with root $R$, that $\zeta_\alpha\not\in s'_\beta$ and $\phi_\alpha|_R = \psi_\beta|_R$ for all $\alpha,\beta$. For every $\alpha$ consider the open set
\[U_\alpha =\core_{s_\alpha}(\varphi_\alpha)\cap \core_{s_\alpha}(\psi_\alpha),\]
as in Claim 2.
Since $L$ is \textit{ccc}, it follows from Claim 2 that we can find an infinite sequence of indices $\alpha_1,\alpha_2,\ldots$ such that
$U_{\alpha_1}\cap U_{\alpha_2}\cap\cdots\cap U_{\alpha_n}\neq\emptyset$
for every $n$. We claim that the sequence of pairs $(C(\varphi_{\alpha_n}),C(\psi_{\alpha_n}))$ is independent, as desired. For this, we must find
$f\in C(\gamma_1)\cap C(\gamma_2)\cap\cdots\cap C(\gamma_n)$
for any choice of $\gamma_i\in \{\varphi_i,\psi_i\}$. Pick $x\in U_{\alpha_1}\cup\cdots\cup U_{\alpha_n}$ and then define $f$ so that
$f|_{s_{\alpha_i}}$ agrees with $\gamma_{\alpha_i}$ for all $i\leq n$, while for $t=\la \xi,p,q\ra\not\in \bigcup_{i=1}^n s_{\alpha_i}$, we declare $f(t) = 1$ when $x(\xi)>p$ and $f(t)=0$ otherwise. Notice that there are no conflicts in this definition because of all the previous refinements on the family. We have that $f\in K$ because $x\in\bigcap_s \core_s(f)$, and $f\in C(\gamma_1)\cap C(\gamma_2)\cap\cdots\cap C(\gamma_n)$.
\end{proof}

Take any nonmetrizable WRN separable compact space $L_0$ (e.g.\ the split interval). Then, the space $L=P(L_0)$
of all regular probability measures on $L_0$ is a separable (so is  $ccc$) nonmetrizable convex WRN compact space.\footnote{The fact that $P(K)$ is WRN whenever $K$ is WRN is a consequence of the characterization of WRN compacta as those compact spaces $K$ for which $C(K)$ is weakly precompactly generated, but also as those compact spaces which can be weak*-embedded into the dual ball of a weakly precompactly generated Banach space; see \cite[Theorem 2.1.4 and 2.1.5]{MaCeThesis}.}
We can apply Theorem \ref{between:2} to get the following, somewhat suprising, result.

\begin{corollary}\label{between:3}
The class $\wrnb^\perp\sm \wrn^\perp$ contains zero-dimensional (necessarily nonmetrizable) 
compact spaces.
\end{corollary}

\section{The orthogonal class of zero-dimensional compacta}\label{ocZDs}

In this section we denote by $\zerodimensional$ the class of zero-dimensional compact spaces. The orthogonal class $\zerodimensional^{\perp}$  can be easily characterized as follows:

\begin{lemma}
$K$ belongs to $\zerodimensional^{\perp}$ if and only if it contains at most countably many different clopens.
\end{lemma}
\begin{proof}
Notice that $K$ has countably many different clopens if and only if its clopen algebra is countable, which in turn is equivalent to the fact that the Stone space $\ult(\Clop(K))$ of its clopen algebra is metrizable.
Bearing in mind that $K$ always maps continuously onto $\ult(\Clop(K))$, we obtain that if $ K \in \zerodimensional^{\perp}$ then it contains at most countably many different clopens.
On the other hand, suppose that $f:K \rightarrow L$ is a continuous map onto a zerodimensional compact space $L$. If $L$ were nonmetrizable, then it would contain uncountably many clopens and, since the preimage of a clopen is a clopen, $K$ would have the same property. Thus, if $K$ contains at most countably many different clopens then it belongs to  $\zerodimensional^{\perp}$.
\end{proof}

From the previous lemma and Lemma \ref{lemhereditarilyCperp} one could give a characterization for the class $\zerodimensional^{(\perp)}$ which turns out to be unsatisfactory in the sense that 
it does not seem to be useful to determine whether this class contains
 nonmetrizable compact spaces.
Let us recall that even a simpler question, whether there are nonmetrizable compact spaces containing
no zero-dimensional nonmetrizable closed subspaces, is somewhat delicate: 
Koszmider \cite{Kosz16} gave the first ZFC example of such a space;
Marciszewski \cite{WM21} gave a consistent example which is Eberlein compact. 

The class  $\zerodimensional^{(\perp)}$ has been studied in the literature; see, e.g., 
\cite[Question 1.1(2)]{Kosz16}, \cite[Question 12 (374)]{GruMoo07}  and \cite[Question 4 (1176)]{Kosz07}. 
It seems to be an open problem whether it is consistent that $\zerodimensional^{(\perp)}$
is trivial (i.e.~it consists solely  of metrizable compacta).
Nevertheless, it may happen (in some models of set theory) 
that $\zerodimensional^{(\perp)}$ contains some
 nonmetrizable Corson compact spaces 
and compact spaces that are not hereditarily separable; on the other hand, 
 this orthogonal class can  contain neither nonmetrizable Eberlein compacta nor 
Rosenthal compacta  (see \cite[Proposition 4.2]{Kosz16}).
It follows that  $\corson \cap \zerodimensional^{(\perp)}$ is trivial  under $MA_{\omega_1}$.
Let us recall a recent result from \cite{Pl20} showing that 
a connected version of Kunen's $L$-space constructed  under CH 
is in $ \corson\cap \zerodimensional^{(\perp)}$.

Let us note  that any Souslin line is a WRN compact space 
(since it is a linearly ordered compact space \cite[Theorem 8.7]{GlasnerMegrelishvili14}) and
belongs to $\zerodimensional^{(\perp)}$, which can be demonstrated
using the argument from  \cite[Proposition 4.2(5)]{Kosz16}). 
We show below, using  the so called split compact spaces introduced by Koszmider,
that nonmetrizable WRN compact spaces in
$\zerodimensional^{(\perp)}$ can be constructed  under Martin's axiom and the negation of CH. 

\begin{definition}
	\label{DEFIsplitcompact}
Let $M$ be a metric compact space, $L$ a compact space, $\kappa$ an ordinal, $\{ r_\xi: \xi < \kappa \}$ a family of distinct points of $L$ and $f_\xi \colon L \setminus \{r_\xi\} \rightarrow M$ a continuous function for every $\xi < \kappa$.
The split $L$ induced by $\{f_\xi : \xi < \kappa \}$ is the subspace $K$ of $L^{\{\ast\}} \times M^\kappa$ consisting of points of the form
\[ \{ x_{\xi,t} : \xi < \kappa, ~t\in M \}\cup \{x_r : r \in L \setminus \{ r_\xi: \xi < \kappa \} \},\]
where
\begin{itemize}
	\item $x_{\xi,t}(\ast)=r_\xi$, $x_{\xi,t}(\xi)=t$ and $x_{\xi,t}(\eta)=f_\eta (r_\xi)$ if $\eta \neq \xi$.
	\item $x_r(\ast)= r$ and $x_r(\xi)=f_\xi(r)$ for all $r\in  L \setminus \{ r_\xi: \xi < \kappa \}$ and every $\xi < \kappa$.
\end{itemize}	
\end{definition}

\bigskip

 The classical split interval is an example of a split compact space of this form. We provide in Theorem \ref{THEOsplittedWRN} a sufficient condition for a split compact space to be WRN. For that purpose, we need to extend the concept of independent functions to functions taking values in any compact space.

\begin{definition}
\label{DefinitionLindependent}
Let $K$ and $M$ be compact spaces. A sequence of functions $f_n \colon K \rto M$ is said to be \emph{$M$-independent} if there exist closed disjoint sets $C,C'$ in $M$ such that  $\left(f_n^{-1}(C), f_n^{-1}(C')\right)_{n \in \N}$ is independent.

 We say that the sequence of functions $f_{\xi_n}\colon L \setminus \{r_{\xi_n} \} \rto M$ is \textit{$M$-independent} if there exist extensions (possibly not continuous)  $g_{\xi_n} \colon L \rto M$ of $f_{\xi_n}$ for each $n \in \N$ such that the sequence $g_{\xi_n}$ is $M$-independent.
\end{definition}

 Notice that $f_{\xi_n}$ is $M$-independent if and only if every extension provides an $M$-independent sequence, i.e.~if $g_n$ and $h_n$ are different extensions of $f_{\xi_n}$ then the sequence $g_n$ is $M$-independent if and only if $h_n$ is $M$-independent. Namely, if  $g_n$ is not $M$-independent then for every closed disjoint sets $C,~C'$ of $M$ there are disjoint finite subsets $S_1, ~S_2$ of $\N$ such that
 \[ \left(\bigcap_{k \in S_1} g_k^{-1}(C) \right)\cap \left(\bigcap_{k' \in S_2} g_{k'}^{-1}(C') \right) = \emptyset.\]
 But then
  \[ \left(\bigcap_{k \in S_1} h_k^{-1}(C) \right)\cap \left(\bigcap_{k' \in S_2} h_{k'}^{-1}(C') \right) \sub \{ r_{\xi_n}: n \in S_1 \cup S_2 \} \]
 is a finite set. Now, a suitable choice of finite sets $S_1' \supseteq S_1$ and $S_2' \supseteq S_2$ shows that $h_n$ is not $M$-independent.

The following lemma is a simple extension of the well-known
Rosenthal Theorem which states that every sequence of  functions $f_n: S\rightarrow [0,1]$ defined on a set $S$ contains a pointwise convergent subsequence or a $[0,1]$-independent subsequence \cite{Rol1}.

\begin{lemma}
	\label{LEMdicextension}
	Let $S$ be a set, $M$ a metric compact space and $f_n\colon S \rto M$ a  sequence of functions.
	Then $f_n$ has a pointwise convergent subsequence or an $M$-independent subsequence. Moreover, $M$-independent sequences do not have pointwise convergent subsequences.
\end{lemma}
\begin{proof}
	Take $q\colon M \rto [0,1]^\N$ an embedding from $M$ into the Hilbert cube and denote by $q_n$ the $n$th-coordinate function of $q$.
	Suppose $f_n$ does not have an $M$-independent subsequence.
	Then, $q_1 \circ f_n$ does not have a $[0,1]$-independent subsequence. By Rosenthal's Theorem, there exists a convergent subsequence  of $q_1 \circ f_n$.
	A standard diagonal argument provides a subsequence $f_{n_k}$ such that
	$\{q_m \circ f_{n_k}\}_{k\in \N}$ converges for every $m \in \N$.
	Thus, $f_{n_k}$ is a convergent subsequence of $f_n$.

	For the last part of the lemma, take $C$ and $C'$ closed disjoint sets witnessing the $M$-independence of a sequence $f_n\colon S \rto M$.
	Let $(f_{n_k})_k$ be any subsequence. Since $\left(f_{n_k}^{-1}(C), f_{n_k}^{-1}(C')\right)_{n \in \N}$ is an independent sequence consisting of compact subsets of $K$, we can take 
\[t \in \bigcap_{k \in \N} \left(f_{n_{2k}}^{-1}(C) \cap f_{n_{2k+1}}^{-1}(C') \right).\] 
Thus, $f_{n_k}(t)$ cannot be a convergent sequence since $C$ and $C'$ are disjoint closed sets.
\end{proof}

\begin{theorem}	
 	\label{THEOsplittedWRN}
 	Let $K$ be the split $L$ induced by $\{f_\xi : \xi < \kappa \}$, where $L$ and $\{f_\xi : \xi < \kappa \}$ are as in Definition \ref{DEFIsplitcompact}. If $\{f_\xi : \xi < \kappa \}$ does not contain $M$-independent sequences and $L$ is WRN, then $K$ is WRN.

\end{theorem}
 \begin{proof}
 	Denote by $\pi_{\ast} \colon K \rto L$ the projection onto the first coordinate (i.e. $\pi_{\ast}(x)=x(\ast)$ for every $x \in K \sub L^{\{\ast\}} \times M^\kappa$) and by $\pi_\xi \colon K \rto M$ the projection onto the coordinate $\xi$, i.e.~$\pi_\xi (x)= x(\xi)$ for every $x \in K$, $\xi < \kappa$.
 	We claim that $\lbrace \pi_\xi : \xi < \kappa \rbrace$ does not contain $M$-independent sequences. Take a sequence $\pi_{\xi_n}$. Since $\{f_\xi : \xi < \kappa \}$ does not contain $M$-independent sequences, by Lemma \ref{LEMdicextension} we may suppose that $f_{\xi_n}$ is pointwise convergent, in the sense that $\left(f_{\xi_n}(x)\right)_{n \in \N, \xi_n \neq x}$ converges for every $x \in L$.
 	Notice that for every $x \in K$,  $\pi_{\xi_n}(x)= f_{\xi_n}(x(\ast))$ for all except at most one $n \in \N$. Thus,
 	the sequence $\pi_{\xi_n}$ is pointwise convergent and therefore it does not contain $M$-independent subsequences due to Lemma \ref{LEMdicextension}. Hence $\lbrace \pi_\xi : \xi < \kappa \rbrace$ does not contain $M$-independent sequences.
 	Since $L$ is WRN, there exists a family $\cF $ of continuous functions  from $L$ to $[0,1]$ separating points and with no independent sequences (Theorem \ref{Gonzalo}). Notice that the family of functions
 	\[ \{ \pi_\xi : \xi < \kappa \} \cup \{ f \circ \pi_{\ast}: f \in \cF \}\]
 	separates the points of $K$.
 	
 	Now take $q \colon M \rto [0,1]^\N$ an embedding from $M$ into the Hilbert cube, with $q_n$ the coordinate functions of $q$.
 	Set $\cF_n=\lbrace \frac{q_n \circ \pi_\xi}{n} : \xi < \kappa \rbrace $ and $\cF'= \bigcup_{n \in \N} \cF_n$.
 	Then, $\cF'$ does not contain independent sequences of functions. It follows that $\cF'\cup \{ f \circ \pi_{\ast}: f \in \cF \}$ is a family of continuous functions which separates the points of $K$ and with no independent subsequences. Therefore, $K$ is WRN again by Theorem \ref{Gonzalo}.
 \end{proof}

\begin{example}
	\label{EXAMFilippov}
	Set $L=[0,1]^2$, $M=\mathbb{S}$, where $\mathbb{S}$ is the unit sphere in $\erre^2$ with the Euclidean metric, $\{ r_\xi: \xi < \kappa \} \sub L$ and
	$f_\xi \colon L \setminus \{ r_\xi \} \rto M$ defined as
	$ f_\xi(x)=\frac{x-r_\xi}{d(x,r_\xi)}$ for every $\xi < \kappa$, where $d$ is the Euclidean distance in $[0,1]^2$. Let $K$ be the split $L$ induced by $\{f_\xi : \xi < \kappa \}$. $K$ is said to be a \emph{Filippov space}.\index{Filippov space} We claim that $K$ is WRN. By Theorem \ref{THEOsplittedWRN}, it is enough to check that every sequence $ f_{\xi_n}$ does not contain an $M$-independent subsequence or, equivalently, that it contains a convergent subsequence.
	However, since $r_{\xi_n}$ is a sequence in $[0,1]^2$, we may suppose without loss of generality that  $r_{\xi_n}$ converges to some $r \in [0,1]^2$. But then notice that
	$f_{\xi_n}(x)= \frac{x-r_{\xi_n}}{d(x,r_{\xi_n})}$ converges to
	$\frac{x-r}{d(x,r)}$ for every $x \neq r$.
	Passing to a subsequence if necessary, we may suppose that the sequence $f_{\xi_n}(r)$ is also convergent. Thus,
	$f_{\xi_n}$ does not contain $M$-independent subsequences and we conclude that $K$ is WRN.
\end{example}

\begin{corollary}
	\label{COROWRNwiththeproperty}
	Under Martin's axiom and the negation of CH there is a WRN nonmetrizable compact space in $\zerodimensional^{(\perp)}$.
\end{corollary}

\begin{proof}
	It is a consequence of Example \ref{EXAMFilippov} and \cite[Theorem 4.5]{Kosz16}, where it is proved that there is a nonmetrizable Filippov space in $\zerodimensional^{(\perp)}$.
\end{proof}

\end{document}